\documentclass[12pt]{article}
\usepackage{amsmath, amsfonts, amsthm, amssymb}
\usepackage{graphicx}
\usepackage{float}
\usepackage{verbatim}

\numberwithin{equation}{section}
\newtheorem{thm}{Theorem}[section]
\newtheorem{lemma}[thm]{Lemma}
\newtheorem{fact}[thm]{Fact}
\newtheorem{cor}[thm]{Corollary}
\newtheorem{definition}[thm]{Definition}

\newtheorem{prop}[thm]{Proposition}

\newtheorem{example}[thm]{Example}

\renewcommand{\geq}{\geqslant}
\renewcommand{\leq}{\leqslant}

\title{Uniform scaling limits for ergodic measures}

\author{Jonathan M. Fraser\\
\emph{School of Mathematics, The University of Manchester,}\\ \emph{Manchester, M13 9PL, UK}\\
\emph{Email: jonathan.fraser@manchester.ac.uk} \\ \\
 Mark Pollicott\\
\emph{Mathematics Institute, Zeeman Building,}\\ \emph{University of Warwick, Coventry, CV4 7AL, UK}\\
\emph{Email: mpollic@maths.warwick.ac.uk}}

\begin{document}
\maketitle

\begin{abstract}
We prove that ergodic measures on one-sided shift spaces are uniformly scaling in the sense of Gavish.  That is, given a shift ergodic measure we prove that at almost every point the scenery distributions weakly converge to a common distribution on the space of measures.  Moreover, we give an explicit description of the limiting distribution in terms of a `reverse Jacobian' function associated with the corresponding measure on the space of left infinite sequences.
\\



\end{abstract}

\section{Introduction}

Given a set or measure one is often interested in studying the fine structure, i.e., properties derived from infinitesimal behaviour.  As such it is important to understand `tangents' and what is currently emerging in the literature on geometric measure theory, ergodic theory and fractal geometry is that understanding the dynamics of the process of `zooming-in' to the tangents is even more valuable.  Some of these ideas go back a long way, in particular to Furstenberg's work in the 60s and 70s, see \cite{Furstenberg60s}, but the techniques and philosophies have recently been applied to great effect, for example see \cite{Furstenberg, Gavish, Hochman, HochmanShmerkin}.  First one defines a process of zooming-in at a point in the support of a given measure.  This may not converge but one is interested in weak accumulation points of this process in the appropriate space of measures.  One drawback of this approach is that one may obtain too many `tangent measures' and not be able to relate them sensibly back to the original measure.  As such one looks to define a measure on the space of measures (commonly referred to as a \emph{distribution}) which best describes which measures are most prevalent during the process of zooming-in. We will make this precise in the context of measures on shift spaces in Section \ref{blowingupmeasures}. Gavish \cite{Gavish} introduced the concept of a measure being `uniformly scaling' if at almost every point in the support of the measure, the zooming-in process generates the same distribution.  As such, uniformly scaling measures are very homogeneous and it turns out that one can make very strong statements about their geometry.  In particular, see \cite{Hochman, HochmanShmerkin, HochmanShmerkin2, kaenmaki} for recent and influential developments in this direction.
\\ \\
In this paper we study the process described above abstractly in the context of ergodic measures on shift spaces.  Our main result is that shift ergodic measures are uniformly scaling in the sense of Gavish \cite{Gavish} and we are able to explicitly describe the generated distribution in terms of the `reverse Jacobian' of the naturally associated measure on the space of left infinite sequences.  Our main result will be given in Section \ref{resultssection} and proved in Section \ref{mainproof}. In Section \ref{resultssectiongibbs} we discuss the simpler setting of ergodic Gibbs measures, where the reverse Jacobian is the classical $g$-function.  We consider some simple examples and finish by proving a Central Limit Theorem for the scaling scenery of ergodic Markov measures.

\subsection{Scaling scenery for measures on shift spaces}   \label{blowingupmeasures}

Let $\mathcal{I} = \{1, \dots, k\}$ be a finite alphabet, $\Sigma^+  = \prod_0^{\infty} \mathcal{I}$ be the space of one-sided sequences over $\mathcal{I}$ and $\sigma$ denote the usual (left) shift map.  Abusing notation slightly we write $x= (x_0 , \dots, x_{n-1}) \in \prod_0^{n-1} \mathcal{I}$ and $x= (x_0, x_1, \dots) \in \Sigma^+$. We equip $\Sigma^+$ with the standard metric defined by
\[
d(x,y) = 2^{-n(x,y)}
\]
where for $x \neq y$, $n(x,y) = \max\{ n \in \mathbb{N} : x_k= y_k \text{ for all $k=0, \dots, n$}\}$. Write $\mathcal P (\Sigma^+)$ for the space of Borel probability measures on $\Sigma^+$ and $\mathcal P_\sigma (\Sigma^+)$ for the space of shift invariant measures from $\mathcal P (\Sigma^+)$.  Equip both these spaces of measures with the weak topology, which can be metrised with either the Levy-Prokhorov or Wasserstein metric for example, and write $\text{spt}(\mu)$ for the support of a measure $\mu$.  Let  $x = (x_0, x_1, x_2, \cdots) \in \Sigma^+$  and $n \geq 1$ and define ``blow up'' maps  $T_{n, x}:  \Sigma^+ \to \Sigma^+ $ by 
$$
T_{n, x} (y_0, y_1, y_2, \cdots)
= (x_0, \cdots , x_{n-1}, y_0, y_1, y_2, \cdots)
$$
i.e., inserting the first $n$ terms from the sequence $x$ at the front of $y$.  We define \emph{cylinder sets} as
\[
[x_0, \dots, x_{n-1} ]_0^{n-1} \ = \ T_{n, x} \Sigma^+ \  = \  \{y \in \Sigma^+
 \hbox{ : } y_i = x_i \hbox{ for } 0 \leq i \leq n-1\} 
\]
and the following definition allows us to blow up $\mu$ on the cylinders containing $x$.
\begin{definition}
The maps $T_{n, x}$ induce a sequence of new measures $\mu_{x,n} \in \mathcal{P}(\Sigma^+)$, which are called \emph{minimeasures}, defined by
$$
\mu_{x,n}(A) =   \frac{ \mu\left(T_{n, x} A \right) }{\mu\left(T_{n, x} \Sigma^+ \right)}
$$
for measurable $A \subset \Sigma^+$, provided $\mu\left(T_{n, x} \Sigma^+ \right)>0$.   This sequence of minimeasures is called the \emph{scaling scenery} of $\mu$ at $x$ and any weak-$*$ accumulation point of the scaling scenery is called a \emph{micromeasure} of $\mu$ at $x$.
\end{definition}

Recently, there has been considerable interest in understanding the limiting behaviour of the scaling scenery and many closely related concepts.  It is perhaps unreasonable to expect the scaling scenery to converge, but one is interested in which minimeasures are most prevalent in the scaling scenery and to this end we define a sequence of \emph{measures on the space of measures} by taking Ces\`aro averages of Dirac measures on the minimeasures along the scaling scenery and then hope that this converges.  Let $\mathcal D (\Sigma^+) = \mathcal P (\mathcal P (\Sigma^+))$ be the space of Borel measures on $\mathcal P (\Sigma^+)$, which we call the space of \emph{distributions}.  

\begin{definition}
The \emph{$N$th scenery distribution} of $\mu$ at $x \in \text{\emph{spt}}(\mu)$ is
\[
\frac{1}{N} \sum_{n=0}^{N-1} \delta_{\mu_{x,n}} \ \in \ \mathcal D (\Sigma^+)
\]
and any weak-$*$ accumulation point of the sequence of $N$th scenery distributions is called a \emph{micromeasure distribution}. 
\end{definition}

It is straightforward to see that any micromeasure distribution at $x$ is supported on the set of micromeasures of $\mu$ at $x$.  Gavish introduced the concept of a measure being \emph{uniformly scaling} if the scenery distributions converge almost everywhere to a common micromeasure distribution.

\begin{definition}
A measure $\mu \in \mathcal P (\Sigma^+)$ is a \emph{uniformly scaling measure} if there exists a distribution $Q \in \mathcal D (\Sigma^+)$ such that at $\mu$ almost every $x \in \Sigma^+$
\[
\frac{1}{N} \sum_{n=0}^{N-1} \delta_{\mu_{x,n}}  \ \to_{w^*}  \ Q.
\]
In this case we say that $\mu$ \emph{generates} the distribution $Q$.
\end{definition}

\subsection{Ergodic measures and the reverse Jacobian}

Let $\Sigma  =\prod_{-\infty}^{\infty}\mathcal{I}$  be the space of infinite two-sided sequences where we write $x= (x_m , \dots, x_{n}) \in \prod_m^{n} \mathcal{I}$ (with $m<n$) and $x= (\dots, x_{-1} ; x_0, x_1, \dots) \in \Sigma$. We also write $\sigma$ for the (invertible) left shift map on $\Sigma$ given by 
$$
\sigma ( \cdots, x_{-2}, x_{-1}; x_0, x_1, x_2, \cdots) 
= ( \cdots,  x_{-1}, x_0; x_1, x_2, x_3, \cdots)
$$
and let $\mathcal P_\sigma(\Sigma)$ denote the space of shift invariant Borel probability measures on $\Sigma$.

\begin{lemma}\label{measures}
There is a natural bijection between the spaces $\mathcal P_\sigma(\Sigma^+)$
and $\mathcal  P_\sigma(\Sigma)$. Moreover, this map is also a bijection between ergodic measures on $\Sigma^+$ and $\Sigma$.
\end{lemma}

\begin{proof}
For the first part we use the (unique) extension of an invariant measure $\mu$ on $\Sigma^+$
to $\Sigma$  given by $\mu([x_{m}, \cdots, x_{n}]_{m}^n)
:= \mu([x_{m}, \cdots, x_{n}]_{0}^{n-m})$ (with $m<n$).  The fact that ergodic measures are paired with ergodic measures is straightforward and omitted.
\end{proof}

Given an ergodic measure $\mu \in \mathcal{P}_\sigma(\Sigma)$, define a sequence of functions $g_n: \Sigma \to [0,1]$ by
$$
g_n(x) = \frac{\mu \big( [x_{-n}, x_{-(n-1)}, \cdots, x_{-1} ]_{-n}^{-1}\big)}{\mu\big([x_{-n},  x_{-(n-1)},  \cdots, x_{-2}]_{-n}^{-2}\big)}
$$
for $x = (x_l)_{l=-\infty}^{\infty} \in \text{spt}(\mu)$ and $g_n(x) = 0$ for $x \in \Sigma \setminus \text{spt}(\mu)$ .  From this sequence of functions we are able to define the reverse Jacobian $g$ which we will need to state our main result, Theorem \ref{main}.
\begin{lemma} \label{reverseg}
The limit $g(x) := \lim_{n \to +\infty} g_n(x)$ exists for $\mu$ almost every $x \in \Sigma$ and, moreover, $g \in L^1(\Sigma, \mu)$.
\end{lemma}

\begin{proof}
  Consider the space of left infinite sequences $\Sigma^-$ and let $\mu^-$ be the push forward of $\mu$ to $\Sigma^-$ via the natural restriction.  Let $\sigma^{-1}$ be the associated right shift and note that $\mu^-$ need not be $\sigma^{-1}$ invariant. Observe that $g_n$ only depends on past coordinates, i.e. $g_n(x) = g_n(x')$ if $x$ and $x'$ are such that $x_n = x_n'$ for all $n<0$, and so for $x \in \text{spt}(\mu)$
\[
g_n(x) = \frac{\mu^- \big( [x_{-n}, x_{-(n-1)}, \cdots, x_{-1} ]_{-n}^{-1}\big)}{\mu^- \circ \sigma^{-1}\big([x_{-n},  x_{-(n-1)},  \cdots, x_{-1}]_{-n}^{-1}\big)}
\]
which is the reciprocal of the Radon-Nikodym derivative $\text{d}(\mu^- \circ \sigma^{-1} )/ \text{d}\mu^- $ with respect to the $\sigma$-algebra generated by the cylinders of length $n$ in $\Sigma^-$.  Even though $\mu^-$ may be singular, the Radon-Nikodym derivative is well-defined because $\sigma^{-1}$ is countable to one, see \cite[Section 10-1]{Parry} and also \cite{ParryWalters}.  It follows from \cite[Proposition 48.1]{Parthasarathy} that $g_n$ converges almost surely to an $L^1$ function $g$, which is the Jacobian of $\text{d}(\mu^- \circ \sigma^{-1}) / \text{d}\mu^- $ with respect to the full Borel $\sigma$-algebra.
\end{proof}

\section{Scaling scenery for ergodic measures} \label{resultssection}

We now wish to make more precise statements about the scaling scenery and to do so we need to introduce some more notation.  Given any word $e = (e_0,e_1, \dots, e_{m-1}) \in \prod_0^{m-1} \mathcal{I}$, and $b>a>0$ the open sets
\[
\mathcal{U}^e ( a, b) : = \left\{\nu \in  \mathcal P(\Sigma^+) \hbox{ : }  \nu( [e_0 \dots e_{m-1} ]_{0}^{m-1} )  \in (a, b) \right\}
\]
generate the weak-$*$ topology on $\mathcal P(\Sigma^+)$ and so determining the value of a distribution on these generating sets determines it uniquely.

\begin{thm} \label{main}
Every ergodic measure $\mu \in \mathcal{P}_\sigma(\Sigma^+)$ is a uniformly scaling measure.  Moreover, for a given ergodic $\mu \in \mathcal{P}_\sigma(\Sigma^+)$ the generated the distribution $Q \in \mathcal{D}(\Sigma^+)$ is characterised as follows. Also write $\mu$ for the associated two-sided ergodic measure from Lemma \ref{measures} and let $g  \in L^1(\Sigma, \mu)$ be given by Lemma \ref{reverseg}.  Then
$$
Q \left(\mathcal{U} ( [e_0 \dots e_{m-1} ]_{0}^{m-1}, a, b) \right) \ = \  \mu \left(
\left\{
y \in \Sigma \hbox{ : } a < \prod_{k=1}^{m}  g \big(\sigma^k(y^-e)\big) < b
\right\}
\right)
$$
for any cylinder $[e_0 \cdots e_{m-1}]_0^{m-1}$ and $a < b$ and where we write $ y^-e = ( \cdots, y_{-2}, y_{-1},  e_0 \cdots e_{m-1}, \cdots )$ observing that since $g$ only depends on past coordinates it does not matter how we complete the sequence to the right.
\end{thm}

The relevance of extensions to bi-infinite sequences in the context of blowing up a set or measure has been observed before.  In particular, see Sullivan's limit diffeomorphisms \cite{Sullivan} and subsequent developments and applications of these ideas \cite{Quas, BedfordFisher2, HochmanShmerkin}.  One heuristic justification is that the positive coordinates give location and the negative coordinates give distortion as one zooms in at that location.  Finally we point out that it is easy to construct invariant non-ergodic measures which are not uniformly scaling.  For such examples the scenery distributions almost surely converge to a common distribution within each ergodic component, but the distributions can vary with ergodic component.

\section{Proof of Theorem \ref{main}} \label{mainproof}

Throughout this section we will write $\mu$ both for the original ergodic measure in $\mathcal{P}_\sigma(\Sigma^+)$ and for the associated ergodic measure in $\mathcal{P}_\sigma(\Sigma)$ from Lemma \ref{measures}. Given a word  $e = (e_0,e_1, \dots, e_{m-1}) \in \prod_0^{m-1} \mathcal{I}$, let us define a sequence of functions
 $g_n^e: \Sigma \to [0,1]$ by 
$$
g_n^e(x) = \frac{\mu \big( [x_{-n}, x_{-(n-1)}, \cdots, x_{-1}; e_0,e_1, \dots, e_{m-1} ]_{-n}^{m-1}\big)}{\mu\big([x_{-n}, x_{-(n-1)}, \cdots, x_{-1}]_{-n}^{-1}\big)}
$$
for $x = (x_l)_{l=-\infty}^{\infty} \in \text{spt}(\mu)$ and $g_n^e(x) =0$ for $x \in \Sigma \setminus \text{spt}(\mu)$.

\begin{lemma}\label{realising}

For $\mu$ almost every $x \in \Sigma$, the sequence $g_n^e(x)$ converges and 
$$
 \lim_{n \to +\infty}g_n^e(x)  \ = \  \prod_{k=1}^{m} g\big(\sigma^k(x^-e)\big)  \ =: \  g^e(x)
$$
where $x^-e = (\dots, x_{-n}, x_{-(n-1)}, \cdots, x_{-1}; e_0,e_1, \dots, e_{m-1}, \dots)$ recalling that $g$ only depends on the past coordinates and so it does not matter how $x^-e$ is filled in to the right.
\end{lemma}

\begin{proof}
We assume that $\mu\big([x_{-n}, x_{-(n-1)}, \cdots, x_{-1}; e_0,e_1, \dots, e_{m-1} ]_{-n}^{m-1}\big) > 0$ for all $n \in \mathbb{N}$ and that $x \in \text{spt}(\mu)$. If this is not the case then the result is trivial and $g^e(x) = 0$.   We have
\begin{eqnarray*}
g_n^e(x) &=& \frac{\mu\big([x_{-n}, x_{-(n-1)}, \cdots, x_{-1}; e_0,e_1, \dots, e_{m-1} ]_{-n}^{m-1}\big)}{\mu\big([x_{-n}, x_{-(n-1)}, \cdots, x_{-1}]_{-n}^{-1}\big)} \\ \\
&=&\prod_{k=1}^{m} \frac{\mu\big([x_{-n}, x_{-(n-1)}, \cdots, x_{-1}; e_0,e_1, \dots, e_{k-1} ]_{-n}^{k-1}\big)}{\mu\big([x_{-n}, x_{-(n-1)}, \cdots, x_{-1}; e_0,e_1, \dots, e_{k-2} ]_{-n}^{k-2}\big)}  \\ \\
&=&\prod_{k=1}^{m} g_{n+k}\big(\sigma^k(x^-e)\big) \\ \\ 
& \to&  \prod_{k=1}^{m} g\big(\sigma^k(x^-e)\big)
\end{eqnarray*}
for $\mu$ almost every $x \in \Sigma$ as $n \to +\infty$ by Lemma \ref{reverseg}.
\end{proof}

\begin{lemma}\label{approx} 
Let $e  \in \prod_0^{m-1} \mathcal{I}$.  Then for any  $\epsilon, \delta  > 0$ we can choose a measurable set
$B \subset \Sigma$ with $\mu(B) <\delta$  and 
$n_0$ such that for $n \geq n_0$
we have 
$$
\sup_{x \in \Sigma \setminus B}\left| g^e_n(x)  - g^e(x)\right| < \epsilon.
$$
\end{lemma}

\begin{proof}
This is an immediate consequence of Egorov's Theorem.
\end{proof}

\begin{lemma}\label{limits}
Fix a measurable set $B \subseteq \Sigma$ and $a,b \in \mathbb{R}$ with $a<b$.  Then for $\mu$ almost every $x\in \Sigma$, as $N \to +\infty$ we have
\[
\frac{1}{N} \, \#\left\{
0 \leq n \leq N-1  \hbox{ : } \sigma^n x \in B
\right\}
\to \mu(B)
\]
and  
\[
\frac{1}{N} \,  \#\left\{
0 \leq n \leq N-1  \hbox{ : }g^e(\sigma^n x) \in (a,b)
\right\}
\to \mu(\{ y \in \Sigma  \hbox{ : }   g^e(y) \in (a,b)\}).
\]
\end{lemma}

\begin{proof}
This follows immediately by applying the Birkhoff ergodic theorem for $\sigma: \Sigma \to \Sigma$ and $\mu$.
\end{proof}

Observe that $\mu_{x,n}([e]_0^{m-1})$ is defined for all $x\in  \text{spt}(\mu)$ and so we can extend it to a function of $x = (x_l)_{l=-\infty}^{\infty} \in \Sigma$ by setting it to zero whenever $(x_l)_{l=0}^{\infty} \notin \text{spt}(\mu) \subseteq \Sigma^+$.   This has the advantage that
\begin{equation} \label{rewritekey}
\mu_{x,n}([e_0, \dots, e_{m-1}]_0^{m-1}) = \frac{\mu[x_{0}, x_{1}, \cdots, x_{n-1},e_0, \dots, e_{m-1}]_0^{m+n-1}}{\mu[x_{0}, x_{1}, \cdots, x_{n-1}]_0^{n-1}}
= g_n^e(\sigma^{n} x).
\end{equation}
In particular, for $x\in \Sigma$ the terms $\mu_{x,n}([e]_0^{m-1})$ depend only on the future coordinates.

We are now in position to prove Theorem \ref{main}.
\begin{proof}
Fix $e = (e_0,e_1, \dots, e_{m-1}) \in \prod_0^{m-1} \mathcal{I}$ and $a,b \in \mathbb{R}$ with $a<b$.  We will estimate the measure of $\mathcal{U}^e(a,b)$ for scenery distributions at generic $x \in \Sigma^+$.  The following fact is stated merely for clarity.

\begin{fact}
If 
$\left(
\frac{1}{N} \sum_{n=0}^{N-1}  \delta_{\mu_{x,n}}
\right) 
\left(
{\mathcal U}^e(a, b)\right)
$
converges for $\mu$ almost every $x\in \Sigma$ to a constant then 
$\left(
\frac{1}{N} \sum_{n=0}^{N-1}  \delta_{\mu_{x,n}}
\right) 
\left(
{\mathcal U}^e(a, b)\right)
$
converges for $\mu$ almost every $x\in \Sigma^+$ to the same constant.
\end{fact}

Let $B \subseteq \Sigma$ and $n_0$ be taken from Lemma \ref{approx} and observe that for $N > n_0$ and for all $x\in \Sigma$ we have
$$
\begin{aligned}
\left(
\frac{1}{N} \sum_{n=0}^{N-1}  \delta_{\mu_{x,n}}
\right) 
\left(
{\mathcal U}^e(a, b)\right)
&= \frac{1}{N} \, \#\left\{
0 \leq n \leq N-1  \hbox{ : }  \mu_{x,n}([e_0, \dots, e_{m-1}]_0^{m-1}) \in (a, b)
\right\}\cr
&= \frac{1}{N} \,  \#\left\{
0 \leq n \leq N-1  \hbox{ : }  g_n^e(\sigma^n x) \in (a, b)
\right\}  \qquad \quad \text{by (\ref{rewritekey})}\cr
&\leq \frac{n_0}{N}  \ +  \ 
\frac{1}{N}  \, \#\left\{
0 \leq n \leq N-1  \hbox{ : } g^e(\sigma^nx) \in (a- \epsilon, b+\epsilon)
\right\}\cr
&\qquad + \  \frac{1}{N} \, \#\left\{
0  \leq n \leq N-1  \hbox{ : } \sigma^n x \in B
\right\}.
\cr
\end{aligned}
$$
Letting $N \to +\infty$  we can apply Lemma \ref{limits} and deduce that for $\mu$ almost every $x\in \Sigma^+$ if $Q \in \mathcal{D}(\Sigma^+)$ is an accumulation point of the scenery distributions at $x$, then 
$$
Q\left(
{\mathcal U}^e(a, b)\right) \leq \mu(\{ y \in \Sigma \hbox{ : } g^e(y)  \in (a-\epsilon, b+\epsilon)\}) + \delta.
$$
A similar argument shows that
$$
Q\left(
{\mathcal U}^e(a, b)\right) \geq \mu(\{ y \in \Sigma \hbox{ : } g^e(y)  \in (a+\epsilon, b-\epsilon) \}) - \delta.
$$
Since $\epsilon, \delta > 0$ are arbitrary and $\mathcal{D}(\Sigma^+)$ is sequentially compact by Prokhorov's Theorem, we deduce that for $\mu$ almost every $x \in \Sigma^+$ the scenery distributions at $x$ converge to a common distribution $Q \in \mathcal{D}(\Sigma^+)$ satisfying
\begin{eqnarray*}
Q\left(
{\mathcal U}^e(a, b)\right) &=& \mu(\{ y \in \Sigma \hbox{ : } g^e(y)  \in (a, b) \}) \\ \\
&=&  \mu \left(
\left\{
y \in \Sigma \hbox{ : } a < \prod_{k=1}^{m}  g\big(\sigma^k(y^-e)\big) < b
\right\}
\right)
\end{eqnarray*}
which completes the proof.
\end{proof}

\section{Scaling scenery for Gibbs measures} \label{resultssectiongibbs}

In this section we specialise to the setting of Gibbs measures and consider some simple examples.  Let $\phi: \Sigma_A^+ \to \mathbb{R}$ be a H\"older continuous potential for a subshift of finite type $\Sigma^+_A$.  A measure $\mu \in \mathcal{P}(\Sigma^+)$ supported on $\Sigma^+_A$ is called a \emph{Gibbs measure} for $\phi$ if there exists constants $C_1, C_2 >0 $ such that
\begin{equation} \label{gibbsbowen}
C_1\ \leq \ \frac{\mu\big([x_0, \dots, x_{n-1}]_0^{n-1}\big)}{\exp\big( \sum_{k=0}^{n-1} \phi(\sigma^kx) - n P(\phi)\big)} \ \leq \ C_2
\end{equation}
for all $x \in \Sigma^+_A$ and all $n \in \mathbb{N}$ and where $P(\phi)$ is the pressure of $\phi$, see \cite{Bowen}.  If $\Sigma_A^+$ is topologically mixing, then there is a unique shift invariant Gibbs measure $\mu = \mu_\phi \in \mathcal{P}_\sigma(\Sigma^+)$ and this Gibbs measure is ergodic. Two very simple examples of shift invariant Gibbs measures are \emph{Bernoulli measures} and \emph{Markov measures}.  We will use these as examples and so briefly recall their definitions.  Let $(p_i)_{i \in \mathcal{I}}$ be a strictly positive probability vector associated to $\mathcal{I}$.  Given the potential $\phi(x) = \log p_{x_0}$ for the full shift, the unique invariant Gibbs measure satisfies
\[
\mu\big([x_0, \dots, x_{n-1}]_0^{n-1}\big) \ = \ p_{x_0} \cdots p_{x_{n-1}}
\]
and is called a \emph{Bernoulli measure}.  One more level of complexity yields Markov measures.  Given an irreducible right stochastic matrix $P = \{p_{i,j}\}_{i,j \in \mathcal{I}}$, let $(\pi_i)_{i \in \mathcal{I}}$ be the unique left invariant eigenvector and define a potential $\phi(x) = \log p_{x_0, x_1}$.  The unique invariant Gibbs measure is called a \emph{Markov measure} and satisfies
\[
\mu\big([x_0, \dots, x_{n-1}]_0^{n-1}\big) \ = \ \pi_{x_0} p_{x_0, x_1} \cdots p_{x_{n-2}, x_{n-1}}.
\]
Markov measures are ergodic and supported on the subshift of finite type given by the transition matrix formed by replacing all non-zero entries in $P$ with 1s .  We will also utilise the theory of Gibbs measures on the two-sided shift space $\Sigma$ which are defined similarly, see \cite{Bowen}.

\begin{lemma} \label{reverseggibbs}
If $\mu$ is an invariant Gibbs measure for a H\"older potential, then $ \psi : = \log g$ is a H\"older potential for the corresponding invariant Gibbs measure on $\Sigma$ given by Lemma \ref{measures} where $g: \Sigma \to [0,1]$ is the (almost everywhere defined) reverse Jacobian function given by Lemma \ref{reverseg}.
\end{lemma}
\begin{proof}
This is a standard result in the general theory of $g$-measures, beginning with Keane in the 70s \cite{Keane}.  The fact that $\psi$ is a potential for the two-sided Gibbs measure $\mu$ is due to Ledrappier \cite{Ledrappier}, see also \cite[Theorem 2.1]{Walters}.
\end{proof}

Theorem \ref{main} and Lemma \ref{reverseggibbs} combine to yield the following result for Gibbs measures.

\begin{cor} \label{maingibbs}
Let $\mu \in \mathcal{P}_\sigma(\Sigma^+)$ be an ergodic Gibbs measure for a H\"older continuous potential $\phi$ defined on a subshift of finite type $\Sigma^+_A$.  Then $\mu$ is uniformly scaling generating a distribution $Q \in \mathcal{D}(\Sigma^+)$.  Moreover, there exists a H\"older potential $\psi : \Sigma_A \to \mathbb{R}$ for the associated two-sided Gibbs measure from Lemma \ref{measures} (which is supported on the corresponding two-sided subshift of finite type $\Sigma_A$) such that
$$
Q \left(\mathcal{U} ( [e_0 \dots e_{m-1} ]_{0}^{m-1}, a, b) \right) \ = \  \mu \left(
\left\{
y \in \Sigma \hbox{ : } a < \exp\left(\sum_{k=1}^{m} \psi \big(\sigma^k(y^-e)\big)\right) < b
\right\}
\right)
$$
for any cylinder $[e_0 \cdots e_{m-1}]_0^{m-1}$ and $a < b$ and where we write $ y^-e = ( \cdots, y_{-2}, y_{-1},  e_0 \cdots e_{m-1}, \cdots )$ observing that since $\psi$ depends only on past coordinates it does not matter how we complete the sequence to the right.
\end{cor}

The following proposition shows that if an ergodic Gibbs measure is fully supported, then the support of the distribution $Q$ is very homogeneous in that all measures in the support are uniformly equivalent. A similar observation which also considers Gibbs measures without full support was made in \cite{FraserPollicott}, where the uniform equivalence was needed to pursue geometric applications.

\begin{prop} \label{gibbsmeasuresequiv}
Let $\mu \in \mathcal{P}(\Sigma^+)$ be a fully supported Gibbs measure for a H\"older continuous potential $\phi$. Then there exists a uniform constant $C \geq 1$ depending only on the potential such that for all measurable $A  \subseteq \Sigma^+$ and all mini- or micromeasures $\nu$ we have
\[
C^{-1} \, \mu(A) \ \leq \  \nu(A) \ \leq \ C \, \mu(A).
\]
\end{prop}

\begin{proof}
It suffices to prove the result for minimeasures because the bounds are clearly preserved under weak convergence to any micromeasure.   Let $\nu = \mu_{x,n}$ be a minimeasure of $\mu$ at $x \in \Sigma^{+}$ at depth $n \in \mathbb{N}$.  It suffices to estimate the measure only for cylinders, so let $y \in \Sigma^+$ and $m \in \mathbb{N}$ define an arbitrary cylinder $[y_0, \dots, y_{m-1}]_0^{m-1} \subseteq \Sigma^+$.  Let the $k$th variation of the potential $\phi$ be defined by 
\[
\text{var}_k(\phi) \ = \ \sup_{x,y \in \Sigma^+} \Big\{ \lvert \phi(x)  - \phi(y) \rvert : x_0 = y_0, \dots,  x_{k-1} = y_{k-1}  \Big\}.
\]
A simple consequence of $\phi$ being H\"older is that it has \emph{summable variations}, i.e.
\[
V(\phi) : = \sum_{k=0}^{\infty} \text{var}_{k}(\phi) < \infty.
\]
We have
\begin{eqnarray*}
&\,& \hspace{-10mm} \frac{\nu\big([y_0, \dots, y_{m-1}]_0^{m-1} \big)}{\mu\big([y_0, \dots , y_{m-1}]_0^{m-1} \big)}\\ \\
& = &  \frac{\mu\big([x_0 , \dots , x_{n-1} , y_0 , \dots  , y_{m-1}]_0^{m+n-1} \big)}{\mu\big([x_0 ,\dots, x_{n-1}]_0^{n-1} \big)\mu\big([y_0, \dots, y_{m-1}]_0^{m-1} \big)} \\ \\
& \leq & \frac{C_2 \, \exp \Big(\sum_{k=0}^{m+n-1}\phi\big(\sigma^k(x_0, \dots , x_{n-1} \, y)  \big) \, - \, (m+n)P(\phi) \Big)}{C_1 \, \exp \Big( \sum_{k=0}^{m-1}\phi(\sigma^k(y)) \, - \, mP(\phi) \Big) \ C_1 \, \exp \Big( \sum_{k=0}^{n-1}\phi(\sigma^k(x))  \, - \, nP(\phi) \Big)} \\ \\
& = & \frac{C_2}{C_1^2} \ \exp \Bigg(\sum_{k=0}^{n-1}\phi\big(\sigma^k(x_0, \dots, x_{n-1} \,  y)  \big) - \sum_{k=0}^{n-1}\phi(\sigma^k(x))  \Bigg) \\ \\
& \leq & \frac{C_2}{C_1^2} \ \exp \Bigg(\sum_{k=0}^{n-1} \text{var}_{n-k}(\phi)\Bigg) \  \leq  \   \frac{C_2}{C_1^2} \, \exp\big(V(\phi) \big) \ < \  \infty.
\end{eqnarray*}
A similar argument going in the opposite direction yields
\[
\frac{\nu\big([\alpha \lvert_n]\big)}{\mu\big([\alpha \lvert_n]\big)} \  \geq  \  \frac{C_1}{C_2^2} \, \exp\big(-V(\phi) \big) \ > \ 0
\]
completing the proof.  The lower bound required $\mu$ to be defined for the full shift, but the upper bound holds more generally.
\end{proof}

Proposition \ref{gibbsmeasuresequiv} shows that all micromeasures of a fully supported Gibbs measure are themselves Gibbs measures for the same potential and a pleasant consequence of this is that there is at most one \emph{invariant} micromeasure for any (invariant or non-invariant) Gibbs measure.
\\ \\
In the simpler setting of locally constant potentials, one can say even more.  In fact, an explicit expression for the generated distribution $Q$ can be derived easily from the definitions.
\begin{example}
Let $\mu \in \mathcal{P}_\sigma(\Sigma^+)$ be a Bernoulli measure.  Then all minimeasures and micromeasures at any point are equal to $\mu$ itself and so $\mu$ is uniformly scaling and generates the distribution $\delta_\mu \in \mathcal{D}(\Sigma^+)$.
\end{example}
\begin{proof}
This follows immediately from the definitions.
\end{proof}
The situation for Markov measures is only slightly more complicated.  Here there are $k$ different measures one can find in the scaling scenery, corresponding to the first level blow ups. For $i \in \mathcal{I}$, let $\mu_i \in \mathcal{P}(\Sigma^+)$ be defined by 
\[
\mu_i(A) =  \frac{ \mu\left(i A \right) }{\mu\left([i]_0^0 \right)}
\]
for a measurable set $A \subseteq \Sigma^+$.

\begin{example} \label{markovexample}
Let $\mu \in \mathcal{P}_\sigma(\Sigma^+)$ be an ergodic Markov measure.  Then $\mu$ is uniformly scaling and generates the distribution
\[
\sum_{i \in \mathcal{I}} \, \pi_i \, \delta_{\mu_i} \in \mathcal{D}(\Sigma^+).
\]
\end{example}
\begin{proof}
Let $i \in \mathcal{I}$, $x \in \Sigma^+$ such that $x_{n-1} = i$ and let $y \in \Sigma^+$ and $m \in \mathbb{N}$ be arbitrary.  Then
\begin{align*}
\mu_{x,n}([y_0, \dots, y_{m-1}]_0^{m-1})  \ &= \  \frac{\mu\big( [x_0, \dots x_{n-1}, y_0, \dots, y_{m-1}]_0^{m+n-1}\big) }{\mu\big( [x_0, \dots, x_{n-1} ]_0^{n-1} \big)} \\ 
&=  \ \frac{\pi_{x_0} \, p_{x_0 x_1} \cdots  p_{x_{n-2} x_{n-1}}  p_{x_{n-1} y_{0}}  \cdots p_{y_{m-2} y_{m-1}} }{\pi_{x_0} \, p_{x_0 x_1} \cdots  p_{x_{n-2} x_{n-1}}} \\ 
 &=  \  p_{x_{n-1} y_{0}}  \cdots p_{y_{m-2} y_{m-1}} \\
&= \  \mu_i([y_0, \dots, y_{m-1}]_0^{m-1}) 
\end{align*}
and so for such $x$ and $n$, $\mu_{x,n} = \mu_i$.  This observation combined with the Birkhoff ergodic 
theorem implies that for $\mu$ almost all $x$ we have
\begin{align*}
\frac{1}{N} \sum_{n=0}^{N-1} \delta_{\mu_{x,n}}   \   & =  \  \sum_{i \in \mathcal{I}} \,  \bigg( \frac{1}{N} \, \sum_{n=0}^{N-1} \textbf{1}_{[i]_0^0}\big(\sigma^n(x) \big)\bigg) \, \delta_{\mu_i} \\ 
&\to_{w^*}  \ \sum_{i \in \mathcal{I}} \,  \bigg( \int_{\Sigma^+} \textbf{1}_{[i]_0^0} \, \text{d} \mu \bigg) \, \delta_{\mu_i} \\
&=  \ \sum_{i \in \mathcal{I}} \, \pi_i \, \delta_{\mu_i}
\end{align*}
completing the proof.
\end{proof}

These results can easily be extended to ``generalised Markov measures'', i.e., the Gibbs measures for locally constant functions.  The simplicity of the result in the case of Markov measures allows us to make more precise statements about the statistical behaviour of the scenery distributions.  For example, we have the following Central Limit Theorem (CLT).

\begin{cor}[Central Limit Theorem]
Let $\mu$ be an ergodic Markov measure and $Q = \sum_{i \in \mathcal{I}} \, \pi_i \, \delta_{\mu_i}$ be the distribution it generates.  Fix a cylinder $[e_0 \cdots e_{m-1}]_0^{m-1}$ and $a < b$ and write $\mathcal{U} = \mathcal{U} ( [e_0 \dots e_{m-1} ]_{0}^{m-1}, a, b)$.  If  $Q(\mathcal{U}) \in (0,1)$, then letting $\sigma^2 = Q(\mathcal{U}) - Q(\mathcal{U})^2>0$ we have
\[
\frac{1}{\sqrt{N}} \, \sum_{n=0}^{N-1} \delta_{\mu_{x,n}} (\mathcal{U})    \ - \  \sqrt{N }\, Q (\mathcal{U})  \ \Rightarrow \  \mathcal{N}(0,\sigma^2)
\]
where $\Rightarrow$ denotes convergence in distribution.
\end{cor}

\begin{proof}
Write $\mathcal{I}_{\mathcal{U}} = \{ i \in \mathcal{I} : \mu_i \in \mathcal{U}\}$.  Example \ref{markovexample} and the classical CLT yield
\begin{eqnarray*}
\frac{1}{\sqrt{N}} \, \sum_{n=0}^{N-1} \delta_{\mu_{x,n}} (\mathcal{U})    \ - \  \sqrt{N} \, Q (\mathcal{U}) &=&  \frac{1}{\sqrt{N}} \, \sum_{n=0}^{N-1} \left( \sum_{i \in \mathcal{I}_{\mathcal{U}} } \textbf{1}_{[i]_0^0}\big(\sigma^n(x) \big) \right) \ - \ \sqrt{N} \, Q (\mathcal{U}) \\ \\
 &\Rightarrow&  \mathcal{N}(0,\sigma^2) 
\end{eqnarray*}
which completes the proof.  We have used the fact that $X_n : = \sum_{i \in \mathcal{I}_{\mathcal{U}} } \textbf{1}_{[i]_0^0}\big(\sigma^n(x) \big) $ is an i.i.d. sequence taking the value 1 with probability $Q(\mathcal{U}) = \sum_{i \in \mathcal{I}_{\mathcal{U}}} \pi_i$ and 0 otherwise.
\end{proof}

Other related statistical results follow similarly, but we omit further details. If one was interested in obtaining a CLT for general Gibbs measures then, inspecting the proof of Theorem \ref{main}, one obtains
\[
\frac{1}{\sqrt{N}} \, \sum_{n=0}^{N-1} \delta_{\mu_{x,n}} (\mathcal{U})    \ - \  \sqrt{N }\, Q(\mathcal{U}) \ = \ 
\frac{1}{\sqrt{N}} \,\sum_{n=0}^{N-1}  \textbf{1}_{(g_n^e)^{-1} (a,b) }\big(\sigma^n(x) \big)\ - \  \sqrt{N }\, Q(\mathcal{U}) .
\]
This expression is more difficult to handle for two reasons.  The first is that it involves an ergodic sum for a sequence of functions and so one needs an analogue of the CLT for Maker's ergodic theorem.  The second and more important reason is that the sequence $X_n: = \textbf{1}_{(g_n^e)^{-1} (a,b) }\big(\sigma^n(x) \big)$ is not i.i.d. and moreover we cannot guarantee that the functions  $\textbf{1}_{(g_n^e)^{-1} (a,b) }$ or even $\textbf{1}_{(g^e)^{-1} (a,b) }$ are H\"older continuous, despite the fact we know (in the Gibbs setting) that $g_n^e$ and $g^e$ are H\"older.  This prevents us from using several standard results on CLTs for ergodic sums, see \cite{ParryPollicott, ParryCoelho}.  In the setting of ergodic non-Gibbs $\mu$, a CLT appears even harder to achieve because we can only guarantee $g^e$ is $L^1$.

\vspace{8mm}

\begin{centering}

\textbf{Acknowledgements}

This work began while JMF was a PDRA of MP at the University of Warwick.  JMF and MP were financially supported in part by the EPSRC grant EP/J013560/1.

\end{centering}

\end{document}